\documentclass[a4paper]{amsart}
\usepackage[utf8]{inputenc}
\usepackage{amssymb}
\usepackage{amsthm}
\usepackage{amsmath}
\usepackage{enumerate}
\usepackage{mathtools}
\usepackage{comment}
\usepackage{todonotes}
\usepackage{xcolor}\mathtoolsset{showonlyrefs}
\usepackage{hyperref}
\usepackage{dsfont}

\usepackage{esint}

\usepackage[a4paper]{geometry}

\newtheorem{thm}{Theorem}[section]
 \newtheorem{cor}[thm]{Corollary}
 \newtheorem{lem}[thm]{Lemma}
 \newtheorem{prop}[thm]{Proposition}

 \theoremstyle{definition}
 \newtheorem{defn}[thm]{Definition}
  
 \theoremstyle{remark}
 \newtheorem{rem}[thm]{Remark}
 \numberwithin{equation}{section}
\DeclareMathOperator{\supp}{supp}
 
\DeclareMathOperator{\dist}{dist}

\newcommand{\R}{\mathbb{R}}

\newcommand{\N}{\mathbb{N}}
\renewcommand{\d}{ \ \mathrm{d}}

\title[]{Zero-extension convergence and Sobolev spaces on changing domains}

\author{Nikita Evseev}
	\address{Analysis on Metric Spaces Unit, Okinawa Institute of Science and Technology Graduate University, 1919-1 Tancha, Onna-son, Okinawa 904-0495, Japan
	}
	\email{nikita2.evseev@gmail.com}
	
\author{Malte Kampschulte}
  \address{Department of Mathematical Analysis, Faculty of Mathematics and Physics, Charles University,
Prague, Czech Republic
  }
  \email{kampschulte@karlin.mff.cuni.cz}
  
\author{Alexander Menovschikov}
 \address{Department of Mathematics, HSE University, Moscow, Russia
 }
 \email{menovschikovmath@gmail.com}

\date{\today}

\begin{document}

\maketitle

\begin{abstract}
 We extend the definition of weak and strong convergence to sequences of Sobolev-functions whose underlying domains themselves are converging. In contrast to previous works, we do so without ever assuming any sort of reference configuration. We then develop the respective theory and counterparts to classical compactness theorems from the fixed domain case. %We do so both for sequences of static domains as well as for the time-dependent case.
Finally, we illustrate the usefulness of these definitions with some examples from applications and compare them to other approaches.
\end{abstract}

% \tableofcontents

\section{Introduction}

In the modern study of partial differential equations and calculus of variations, one is often faced with problems where finding the domain itself is part of the solution. This includes questions of domain optimization, where one asks about the optimal domain on which to solve a problem \cite{bucur2005variational}, but in particular also problems such as fluid structure interaction (see e.g. \cite{muhaExistenceWeakSolution2013,lengelerWeakSolutionsIncompressible2014,benesovaVariationalApproachHyperbolic2020}), where the evolution of one part of the problem (e.g.\ the evolution of a deformable solid) determines a domain on which another part of the problem is defined (e.g.\ a fluid that is displaced by the solid), which then in turn couples back into the first one via boundary conditions.

When one studies this kind of problems, one is then often dealing with sequences of approximative solutions and their convergence. For this, the first step is to consider the convergence of domains. This, in itself is already well studied. But then given a converging sequence of domains $\Omega_i\subset \R^n$ and a limit $\Omega_0 \subset \R^n$, one then has a matching sequence of functions $u_i: \Omega_i \to \R$, usually in some problem-appropriate space, and needs to identify a correct limit function $u_0:\Omega_0 \to \R$.

As this is a common issue in many applications, there have been many proposed solutions. These solutions commonly revolve around defining a common reference configuration $\Omega_r$ and maps $\Phi_i: \Omega_r \to \Omega_i$ which make it possible to pull back the problem to that of convergence of functions $u_i \circ \Phi_i:\Omega_r \to \R$ on the single reference domain. 

Whenever there is some sort of natural reference configuration, this is clearly the right approach. But in many cases, this reference is somewhat arbitrary (indeed, in continuum mechanics these are often called ``arbitrary Lagrangian-Eulerian'' (ALE) approaches, see e.g.\ \cite{desjardins2001weak,muha2013nonlinear,alphonseAbstractFrameworkParabolic2015}) and in more complicated situations can lead to artificial difficulties stemming mainly from bad choices of the maps $\Phi_i$, which can relate points of the different $\Omega_i$ that should not have anything in common and which generally preclude changes in topology.

In this paper we thus propose an alternative approach, which instead builds directly on the assumption that the only thing that the different domains $\Omega_i$ have in common is that they are subsets of the same $\R^n$. As any notion of strong convergence should start with that of a distance, we thus need a way of comparing two functions directly, in a way that on the common parts of their domains $\Omega_i \cap \Omega_j$ behaves like the usual distance in $W^{1,p}$, while at the same time taking into account those parts where their domain differs.

An obvious approach to this then is using extension theorems to compare extensions on $\Omega_i \cup \Omega_j$ instead. However, while extension theorems from $W^{1,p}(\Omega_i)$ to $W^{1,p}(\R^n)$ exist and are well studied, using them results in its own issues, as they are strongly dependent on regularity of the boundaries $\partial \Omega_i$ and often arbitrary choices are involved in constructing an extension.

However, it turns out that all those difficulties can be circumvented by considering the trivial extension of any function by zero. This is clearly unique and intrinsically ``stable'' under changing domains. It has of course the downside that such an extension of a Sobolev-function is no longer a Sobolev function. Yet, in order to define a distance, this is actually not needed, as instead of the derivative of the extension, it is enough to consider the extension of the derivative. All this then results in our definition of ``zero-extension'' convergence. In addition to the resulting strong convergence, we are also able to define a matching weak convergence using a duality approach.

Now, with such a definition in hand, one can then consider the question of its usefulness in practice. We claim that this is indeed the case. In fact, we can prove analogous statements to the usual theorems about Sobolev-spaces on fixed domains, including standard compactness results such as Banach-Alaoglu and Rellich-Kondrachov type theorems. This allows problems on changing domains to be handled in almost precisely the same way to those on fixed domains, which should prove a great help for applications, as we will illustrate.

The structure of this article is thus as follows. In Section 2 we will introduce our basic definitions and some of the resulting properties for a static domain. Section 3 will be devoted to proving analogues of the common compactness theorems required for applications. In Section 4 we will quickly sketch some ideas on how to deal with boundary values (which in more detail will be the study of a future paper), before we apply all this to some examples in Section 5. Finally in Section 6, we will compare our definitions to other common approaches to this kind of problem.

\subsection*{Acknowledgments}
M.K.\ is supported by the ERC-CZ Grant CONTACT LL2105 funded
by the Ministry of Education, Youth and Sport of the Czech Republic as well as by the Czech Science Foundation (GA\v{C}R) under grant No.\ 23-04766S.

The authors are grateful to Professor Alexander Ukhlov for fruitful discussions and valuable comments.

\section{Strong convergence on a sequence of domains}

Our general setup will be that of a sequence of Banach spaces $(X_i)_{i \in \N}$ and on that sequence we will consider a sequence of elements $(u_i)_{i \in \N}$, $u_i \in X_i$ for all $i \in \N$. For such sequences we can introduce an abstract definition of convergence.

\begin{defn}\label{definition:abstractconv}
     Let $(X_i)_{i \in \N}$ be a sequence of Banach spaces. Consider a sequence $(u_i)_{i \in \N}$ such that $u_i \in X_i$ for all $i \in \N$. 
     We say that $(u_i)_{i \in \N}$ converges to a \textit{limit} $u$ that belongs to a \textit{limit Banach space} $X$ in a (strong) abstract sense if there exists a relation (convergence) $(u_i)_{i \in \N} \mapsto u$, such that
     \begin{enumerate}
         \item $\|u_i\|_{X_i} \to \|u\|_{X} \text{ as } i \to \infty$, \\
         \item the element $u\in X$ is unique, \\
         \item any subsequence of $(u_{i_k})_{k \in \N}$ converges to the same limit $u$.
     \end{enumerate}
\end{defn}

Of course, such a definition is too general, and we introduce it just to emphasize the general settings of our research:
a sequence of Banach spaces, a sequence of vectors and the notion of convergence.
In fact, here we are talking about the convergence of a sequence of pairs $(u_i, X_i)$ to some pair $(u,X)$, i.e. taking a limit both with respect to functions and with respect to spaces.
The idea behind these three properties of the introduced abstract notion is to provide such a definition of convergence that, from the one side, can be useful for constructing the corresponding topological space and, from the other, be applicable to exact applications.    

Now we move on to considering more precise constructions. Thorough the paper we will consider a sequence of domains $\Omega_i \subset \R^n$ that converges to $\Omega \subset \R^n$ in the \textit{Hausdorff-sense}.
Recall that the Hausdorff distance is defined as
\[
d_{H}(\Omega_1, \Omega_2) := \max\{\sup\limits_{x\in\Omega_1} d(x, \Omega_2), \sup\limits_{y \in \Omega_2} d(\Omega_1, y)  \}, 
\]
then $\Omega_i$ converges to $\Omega$ in the Hausdorff-sense if
$$
d_{H}(\Omega_i, \Omega) \to 0 \quad \text{ as } i \to \infty.
$$
We will use the notation of this convergence for sequences of domains unless otherwise stated. 
Although most of the propositions of this article also hold for more general notions of convergence (e.g., \ convergence a.e.\ or in measure), we restrict ourselves for simplicity. Moreover, Hausdorff convergence is a suitable notion of convergence for the applications that we dealing with.

\subsection{Zero-Extension convergence}

As mentioned, our main object is a sequence of Sobolev spaces, and we want to focus on those definitions that keep common parts invariant, 
i.e.\ if one considers a function $u$ supported in $U \subset \bigcup_{i>i_0} \Omega_i \cap \Omega$, then this convergence of domains should be in some sense invisible and the distance between $u$ considered as a function on $\Omega_i$ and as a function on $\Omega$ should be zero. The way to do this is by using extensions. While extending a function $u_i \in W^{1,p}(\Omega_i)$ to a function $\hat{u}_i \in W^{1,p}(\R^n)$ is usually possible, the way this extension looks like depends strongly on even minor changes of $\Omega_i$. 
Instead, it turns out there is a much simpler way that, even though it does not result in a differentiable function on $\R^n$, still suffices for all applications.

For any measurable set $\Omega$ and any function $u \in L^p(\Omega)$, we can define its \textit{zero-extension} $\tilde{u}$ as
\begin{align*}
 \tilde{u}(x) = \begin{cases} u(x) &\text{ if } x\in \Omega\\ 0 & \text{ otherwise.} \end{cases}
\end{align*}
In particular we have $\widetilde{u} \in L^p(\R^n)$ with $\|\tilde{u}\|_{L^p(\R^n)} = \|u\|_{L^p(\Omega)}$. Though technically this operator (denoted also as $(.)^\sim$)\footnote{We use a similar notation convention as is usual for the Fourier-transform, i.e.\ as a diacritic $\tilde{u}$ when applied to individual symbols and as a superscript $(.)^\sim$ for longer expressions.} depends on $\Omega$, we will generally not denote this dependence as it will be clear from the domain of definition of the respective function.

\begin{defn}[Zero-extension convergence]
 Let $(\Omega_i)_{i \in \N}$ be a sequence of domains that converges to a domain $\Omega$ in the Hausdorff-sense and let $u_i \in W^{k,p}(\Omega_i)$ for each $i\in\N$, $u \in W^{k,p}(\Omega)$. We say that $u_i \to u$ in the $W^{k,p}$-\emph{zero-extension sense} ($u_i \xrightarrow{ZE} u$) if
 \begin{align*}
 \sum_{|\alpha|\leq k}\| (D^\alpha u_i)^{\sim} - (D^\alpha u)^\sim\|_{L^p(\R^n)} \to 0 \text{ as } i\to\infty.
 \end{align*}
\end{defn}

This convergence is in consistence with an abstract convergence (Definition \ref{definition:abstractconv}). 
The first and the third properties of abstract convergence can be easily directly verified. Let us show the uniqueness of the limit, which is a direct consequence of the following:

\begin{prop}[Uniqueness]\label{lem:supp}
Let $(\Omega_i)_{i \in \N}$ be a sequence of domains that converges to a domain $\Omega$ in the Hausdorff-sense. 
Let $u_i \in L^{p}(\Omega_i)$, $v \in L^{p}(\R^n)$ and
 \begin{align}
   \| \tilde u_i - v \|_{L^p(\R^n)} \to 0 \text{ for } i \to \infty.
 \end{align}
Then there exists a unique $u \in L^{p}(\Omega)$ whose zero-extension is given by $v$.
\end{prop}
\begin{proof}
Let $\Omega^\varepsilon$ be a  $\varepsilon$-neighborhood of $\Omega$. 
Then for all sufficiently large $i$, by the Hausdorff-convergence, we have $\Omega_i \subset \Omega^\varepsilon$ and thus 
$\supp u_i \subset \overline{\Omega^\varepsilon}$. But then $\supp v \subset \overline{\Omega^\varepsilon}$ for all $\varepsilon > 0$, which implies $\supp v \subset \overline{\Omega}$. 
This means that $v|_{\Omega}$ is the (unique) function in $L^p(\Omega)$ whose zero-extension is given by $v$. We can thus safely define $u := v|_{\Omega}$.
\end{proof}

The following proposition establishes a Cauchy criterion for zero-extension convergence.

\begin{prop}[Cauchy criterion]\label{thm:CauchyZE}
 Let $(\Omega_i)_{i \in \N}$ be a sequence of domains that converges to a domain $\Omega$ in the Hausdorff-sense and let $u_i \in W^{k,p}(\Omega_i)$ for each $i\in\N$. Assume
 \begin{align} \label{eq:extCauchy}
   \sum_{|\alpha\leq k|}  \| (D^\alpha u_i)^\sim - (D^\alpha u_j)^\sim\|_{L^p(\R^n)} \to 0 \text{ for } i,j \to \infty.
 \end{align}
 Then there exists a unique $u \in W^{k,p}(\Omega)$ 
 such that $u_i \to u$ in the zero-extension sense. 
 Conversely, for any $u \in W^{k,p}(\Omega)$ such that $u_i \to u$ in the zero-extension sense, we have \eqref{eq:extCauchy}.
\end{prop}

\begin{proof}
 Assume that \eqref{eq:extCauchy} holds. 
 Then for each multi-index $\alpha$,  $|\alpha|\leq k$, the sequence $v_i^\alpha := (D^\alpha u_i)^\sim$ is a Cauchy sequence in $L^p(\R^n)$ and thus has a unique limit $v^\alpha \in L^p(\R^n)$. 
By Proposition \ref{lem:supp},  $v_i^\alpha |_\Omega$  is the (unique) function in $L^q(\Omega)$ whose zero-extension is given by $v^\alpha$.
We can thus safely define $u := v^0|_{\Omega}$.
 
 What is left to show is that $v^\alpha |_{\Omega} = D^\alpha u$. 
 For this fix $\varphi \in C_0^\infty(\R^n)$ with $\supp \varphi \subset \Omega$. Then for all $i$ large enough, we have $\supp \varphi \subset \Omega_i$. With this, on the one hand
 \begin{align*}
  (-1)^{|\alpha|} \int_{\Omega_i} u_i D^\alpha  \varphi \,\d x = \int_{\Omega_i} D^\alpha u_i \varphi \,\d x = \int_{\R^n} v^\alpha_i \varphi \, \d x \to \int_{\R^n} v^\alpha \varphi \, dx 
  =\int_{\Omega} v^\alpha |_{\Omega} \varphi \, \d x
 \end{align*}
 and on the other hand
 \begin{align*}
   \int_{\Omega_i} u_i D^\alpha \varphi \, \d x = \int_{\R^n} v^0_i D^\alpha \varphi \, \d x \to  \int_{\R^n} v^0 D^\alpha \varphi \, \d x = \int_{\Omega} u D^\alpha \varphi \, \d x
 \end{align*}
 which implies that $v^\alpha |_{\Omega} \in L^p(\Omega)$ is indeed the weak derivative of $u$.
 
 The converse is a direct consequence of the definition.
\end{proof}

The following lemma shows that zero-extension convergence supports the basic properties of Sobolev space like approximation by smooth functions.

\begin{lem} \label{lem:smoothApprox}
 Let $(\Omega_i)_{i \in \N}$ be a sequence of domains that converges to a domain $\Omega$ in the Hausdorff-sense and let $u_i \in W^{k,p}(\Omega_i)$ for each $i\in\N$, $u \in W^{k,p}(\Omega)$. Assume that $u_i$ converges to $u$ in the zero-extension sense. Then there exists a sequence $v_i \in C^\infty(\Omega_i)$ such that $\|v_i-u_i\|_{W^{1,p}(\Omega_i)} \to 0$ and $v_i \to u$ in the zero-extension sense.
\end{lem}

\begin{proof}
 For any fixed $i$, we can find a smooth approximation $v_i \in C^\infty(\Omega_i)$ by a standard convolution argument. This approximation can be chosen in such a way that $\|v_i-u_i\|_{W^{1,p}(\Omega_i)} < \frac{1}{i}$. But then also
 \begin{equation*}
  \|\tilde v_i - \tilde u\|_{W^{1,p}(\R^n)} \leq  \|v_i-u_i\|_{W^{1,p}(\Omega_i)} +  \|\tilde u_i - \tilde u\|_{W^{1,p}(\R^n)} \to 0.\qedhere
 \end{equation*}
\end{proof}

\begin{rem}
Going back to the reasoning behind the definition of abstract convergence, we can consider topology on pairs $(u_i, \Omega_i)$ of Sobolev functions and corresponding domain of definitions. This topology can be metrizable on $\biguplus_\Omega W^{k,p}(\Omega)$ (disjoint union over all bounded domains) by
 \begin{align*}
  d( (\Omega_1,u_1) ,(\Omega_2,u_2) ) := d_{\text{dom}}( \Omega_1,\Omega_2) + \sum_{|\alpha|\leq k}  \| (D^\alpha u_1)^\sim - (D^\alpha u_2)^\sim\|_{L^p(\R^n)},
 \end{align*}
 where $d_{\text{dom}}$ is a metric on the set of all bounded domains (e.g. Hausdorff distance).
 
 But it should be kept in mind that the second term is not a norm. In particular, two functions should only be considered identical if they share the same domain, so e.g.\ $0|_{\Omega_1}$ and $0|_{\Omega_2}$ are different, even if $\| (0|_{\Omega_1})^\sim - ( 0|_{\Omega_2})^\sim \|_{L^p(\R^n)} = 0$.
\end{rem}

\subsection{Weak convergence}

Having now a solid definition of strong convergence, we can define a corresponding notion of weak convergence. 
For this, we rely on a general notion of duality. In an abstract picture, this looks as follows:

\begin{defn}[Abstract weak and weak* convergence]\label{Abstractweak}
Let $(X_i)_{i \in \N}$ be a sequence of Banach spaces. 
Consider a sequence $(u_i)_{i \in \N}$, such that $u_i \in X_i$ for all $i \in \N$. 
Assume that there exists a limit space $X$ and we have a notion of abstract strong convergence.
Denote by $X_i^*$ and $X^*$ their respective topological duals. 
Similarly, assume that there exists a notion of abstract strong convergence for sequences in $(X_i^*)_{i \in \N}$.  %$(v_i)_{i \in \N}$, $v_i \in X_i^*$, $v_i \to v$ and $v \in X^*$.
 \begin{enumerate}
  \item We say that sequence $u_i \in X_i$, converges in \textit{weak  $X$-abstract sense} to $u\in X$, or $u_i \rightharpoonup u$ if
  \begin{align*}
   \langle u_i, v_i \rangle_{X_i\times X_i^*} \to \langle u,v\rangle_{X\times X^*}
  \end{align*}
  for all sequences $v_i \in X_i^*$ and $v \in X^*$ such that $v_i \to v$.
  
  \item We say that sequence $v_i \in X_i^*$, converges in \textit{weak* $X$-abstract  sense} to $v\in X^*$, or $v_i \stackrel{\ast}{\rightharpoonup} v $ if
  \begin{align*}
   \langle u_i, v_i \rangle_{X_i\times X_i^*} \to \langle u,v\rangle_{X\times X^*}
  \end{align*}
  for all sequences $u_i \in X_i$ and $u \in X$ such that $u_i \to u$.
 \end{enumerate}
\end{defn}

For the sequence of Sobolev spaces $W^{k,p}(\Omega_i)$, one can give an equivalent and more explicit definition. 
For simplicity, we will focus on the cases $p \in (1,\infty)$ where they are reflexive. 
Namely, for zero-extension convergence we have the following proposition.

\begin{prop}[Characterization of weak zero-extension convergence]\label{Weakeqv}
    Let $(\Omega_i)_{i \in \N}$ be a sequence of domains that converges to a domain $\Omega$ in the Hausdorff-sense and let $u_i \in W^{k,p}(\Omega_i)$ for each $i\in\N$, 
    $u \in W^{k,p}(\Omega)$. 
    Then $u_i$ converges in the weak $W^{k,p}$-abstract sense with respect to zero-extension convergence ($u_i \xrightharpoonup{\text{ZE}} u$), if
 \begin{align*}
   \int_{\R^n} (D^\alpha u_i)^{\sim} \psi_\alpha \d x &\to \int_{\R^n} (D^\alpha u)^{\sim} \psi_\alpha \d x
 \end{align*}
 for all $|\alpha|\leq k$ and all $\psi_\alpha \in L^{p'}(\R^n)$, $p' = p/(p-1)$.
\end{prop}

\begin{proof}
We prove this proposition for the case $k=1$. The general case follows by induction.

To do this, we can use the Riesz Representation theorem. Consider a functional $l \in (W^{1,p}(\Omega))^*$. For $1< p < \infty$, we can isometrically represent it as a pair $g = (\psi, \phi) \in L^{p^*}(\Omega; \mathbb{R}^{n+1})$, $\psi \in L^{p^*}(\Omega)$, $\phi \in L^{p^*}(\Omega; \mathbb{R}^n)$. Then, as a corresponding convergence on dual spaces, we consider a zero-extension convergence on $L^{p^*}(\Omega; \R^{n+1})$ since the convergence preserves under isometry.

In accordance to Definition \ref{Abstractweak}, the function $u \in W^{1,p}(\Omega)$ is a weak limit of $(u_i)_{i \in \N}$ if
$$
    \langle l_i, u_i \rangle_i \to \langle l, u \rangle|_{i \to \infty}
$$
for all sequences $l_i \in (W^{1,q}(\Omega_i))^*$ and all linear functionals $l \in (W^{1,p}(\Omega))^*$ such that $l_i \xrightarrow{ZE} l$.

By the Riesz Representation Theorem this is equivalent to
\begin{align*}
  \int_{\R^n} \tilde{u}_i \phi_i \d x &\to \int_{\R^n} \tilde{u} \phi \d x \\
  \int_{\R^n} (\nabla u_i)^{\sim} \psi_i \d x &\to \int_{\R^n} (\nabla u)^{\sim} \psi \d x
\end{align*}
for all sequences $g_i =(\phi_i, \psi_i)\in L^{p'}(\Omega_i;\R^{n+1})$ and all $g \in L^{p'}(\Omega; \R^{n+1})$, such that $g_i \xrightarrow{ZE} g$.

The above convergence does not depend on the choice of the sequence $g_i$. Let $(\phi_i)_{i \in \N}$ and $(\theta_i)_{i \in \N}$ be two sequences that converge to the same zero-extension limit $\phi \in L^{p^*}(\Omega)$. Let also
$$
\int_{\R^n} \tilde{u}_i \phi_i \d x \to \int_{\R^n} \tilde{u} \phi \d x.
$$
Then
\begin{align*}
    \int_{\R^n} \tilde{u}_i \theta_i \d x &= \int_{\R^n} \tilde{u}_i \theta_i \d x + \int_{\R^n} \tilde{u}_i \phi_i \d x - \int_{\R^n} \tilde{u}_i \phi_i \d x = \int_{\R^n} \tilde{u}_i (\theta_i -\phi_i) \d x + \int_{\R^n} \tilde{u}_i \phi_i dx \\ 
    &\leq \bigg(\int_{\R^n}|\tilde{u}_i|^p \d x\bigg)^{1/p} \bigg(\int_{\R^n} |\theta_i - \phi_i|^{p^*} \d x\bigg)^{1/p^*} + \int_{\R^n} \tilde{u}_i \phi_i \d x \to \int_{\R^n} \tilde{u} \phi \d x.
\end{align*}

Note that, if $\Omega_i \to \Omega$ even in the sense $\chi_{\Omega_i} \to \chi_\Omega$ a.e. on $\mathbb{R}^n$ and $1< p < \infty$, then for every linear functional on $l \in (W^{1,p}(\Omega))^*$ we can take a sequence $(l_i)_{i \in \N}$, $l_i \in (W^{1,p}(\Omega_i))^*$ for all $i\in\N$, such that $l_i \xrightarrow{ZE} l|_{i \to \infty}$. Indeed, again we identify linear functionals on $(W^{1,p}(\Omega_i))^*$ with functions $g_i = (\psi_i, \phi_i) \in L^{p^*}(\Omega_i; \mathbb{R}^{n+1})$. Then the sequence $g_i = (\widetilde{\psi}|_{\Omega_i}, \widetilde{\phi}|_{\Omega_i})$ converges to $g = (\psi, \phi) \in L^{p^*}(\Omega; \mathbb{R}^{n+1})$ in the zero-extension sense.

These two facts lead us to the conclusion that Definition \ref{Abstractweak} is equivalent to the statement in the proposition.
\end{proof}

In turn, the above equivalent definition of weak zero-extension convergence directly implies that we have the \textit{reflexivity property} in a sense that weak and weak* zero-extension convergences coincide. Similarly to before we also have uniqueness of limits.

\begin{prop}
    Let $(\Omega_i)_{i \in \N}$ be a sequence of domains that converges to a domain $\Omega$ in the Hausdorff-sense and let $u_i \in W^{k,p}(\Omega_i)$ for each $i\in\N$, $u \in W^{k,p}(\Omega)$. If $(u_i)_{i \in \N}$ converges in the weak $W^{1,p}$-zero-extension sense, then such limit is unique.
\end{prop}

\begin{proof} Again, we provide a proof only for the case $k=1$, the general case easily follows.
    Assume that there exist two zero-extension weak limits, $u_i \xrightharpoonup{\text{ZE}} u$, $u_i \xrightharpoonup{\text{ZE}} v$, with $u, v \in W^{1,p}(\Omega)$. 
    Then, by Proposition \ref{Weakeqv} for all $g=(\phi, \psi)\in L^{p'}(\R^n;\R^{n+1})$,
    \begin{align*}
        \int_{\R^n} \tilde{u} \phi \d x = \lim\limits_{i \to \infty}\int_{\R^n} \tilde{u}_i \phi \d x &= \int_{\R^n} \tilde{v} \phi \d x  \\
        \int_{\R^n} (\nabla u)^{\sim} \psi \d x = \lim\limits_{i \to \infty}\int_{\R^n} (\nabla u_i)^{\sim} \psi \d x &= \int_{\R^n} (\nabla v)^{\sim} \psi \d x.
    \end{align*}

    Therefore,
    $$
        \|u-v\|_{W^{1,p}(\Omega)} \leq 
        \sup\limits_{\|\phi\|_{L^{p'}(\Omega)} = 1} \int_{\Omega}  |u-v||\phi| \, \d x + \sup\limits_{\|\psi\|_{L^{p'}(\Omega; \mathbb{R}^n)} = 1} \int_{\Omega}  |\nabla u - \nabla v||\psi|  \, \d x = 0,
    $$
    which finishes the proof.
\end{proof}

\subsection{Weak*-compactness and Banach-Alaoglu}

We begin our studies of compactness with compactness in the weak* topology. Specifically, we aim to create a version of the Banach-Alaoglu theorem. Although this is a quite abstract theorem that is not restricted only to Sobolev-spaces, its usual formulation still requires the structure of a Banach space, which we do not have. Thus we cannot rely on simply citing a general formulation from other sources and have to do the proof ourselves. However, the underlying ideas follow the classical proof closely.

\begin{thm}[Banach-Alaoglu Theorem]\label{BA-internal}
Let $(\Omega_i)_{i \in \N}$ be a sequence of domains that converges to a domain $\Omega$ in the Hausdorff-sense and let $u_i \in W^{k,p}(\Omega_i)$ for each $i\in\N$. If a sequence $(u_i)_{i \in \N}$ is bounded: $\sup\limits_{i \in \mathbb{N}}\|u_i\|_{W^{1,p}(\Omega_i)} < C$, 
then there exists $u \in W^{1,p}(\Omega)$ and a subsequence (not relabeled) such that $u_{i}$ converges in the weak
%\noteNikita{We need to discuss this notation}
 $W^{1,p}$-zero-extension sense to $u$. 
\end{thm}
\begin{proof}
By the assumption 
\[
\|\widetilde{u}_i\|_{L^p(\mathbb R^n)} + \sum_{j=1}^{n}\|(\partial_ju_i)^\sim\|_{L^p(\mathbb R^n)}  \leq C.
\]
First, we prove that there is $u\in L^p(\Omega)$ such that $u_i \xrightharpoonup{\text{ZE}} u$ converges as $L^p$-functions.
By the weak compactness of $L^p(\mathbb R^n)$  there is a subsequence (not relabeled) $(\widetilde{u}_i)_{i\in\mathbb N}$ which weakly converges to some $w\in L^p(\mathbb R^n)$.
In particular for any $\varepsilon > 0$ and $\varphi\in C^\infty_0(\mathbb R^n\setminus\overline{\Omega^\varepsilon})$ we have
\[
0 = \lim_{i\to\infty}  \int_{\mathbb R^n} \varphi(x)\widetilde{u}_i(x) \ \textrm{d}x =  \int_{\mathbb R^n} \varphi(x)w(x) \ \textrm{d}x.
\]
The last implies that $\supp w \subset \overline{\Omega}$. Then we define $u:= w|_{\Omega}$.
Let $v_i\in L^p(\Omega_i)$ and $v_i \xrightarrow{ZE} v$ converge as $L^p$-functions.
Then 
\[
\int_{\Omega_i} v_i(x)u_i(x) \ \textrm{d}x = \int_{\mathbb R^n} \widetilde{v}_i(x)\widetilde{u}_i(x) \ \textrm{d}x \to \int_{\mathbb R^n} \widetilde{v}(x)w(x) \ \textrm{d}x =  \int_{\Omega} v(x)u(x) \ \textrm{d}x.
\]
So we proved that  $(u_i)_{i \in \N}$ converges in weak $L^p$-zero-extension sense to $u$. 
In the same manner, we obtain $\partial_ju_i \xrightharpoonup{\text{ZE}} w_j$ as $L^p$-functions. 
It remains to show that $w_j = \partial_j u$. 

Let $\varphi\in C^\infty_0(\Omega)$. Note that there is a number $N$ such that $\supp\varphi \subset \Omega_i$ for all $i>N$. Then we derive
\[
\int_{\Omega} \varphi(x)w_j(x) \ \textrm{d}x 
= \lim_{i\to\infty} \int_{\Omega_i} \varphi(x)\partial_j u_i(x) \ \textrm{d}x
= \lim_{i\to\infty} \int_{\Omega_i} \frac{\partial\varphi}{\partial x_j}(x) u_i(x) \ \textrm{d}x
= \int_{\Omega} \frac{\partial\varphi}{\partial x_j}(x) u(x) \ \textrm{d}x.\qedhere
\]
\end{proof}

\section{Strong compactness and Rellich-Kondrachov} \label{sec:compactness}

Next, we want to consider compactness in the strong topology. 
This relies more on the specific structure inherent to Sobolev spaces. In particular, the proof relies on the Sobolev inequality that usually involves integration along lines. 
As such, it is not surprising that, at this point, we require additional regularity of the boundary in order to make more general statements.

\subsection{Poincar\'e-sequences of domains.}
For $1\leq p < n$, a ball $B\subset\R^n$, and a function $u\in W^{1,p}(B)$ the following Poincar\'e inequality holds:
\begin{equation}\label{eq:PI}
\bigg( \fint_{B} |u(x) - u_B|^{p^*} \ \textrm{d}x\bigg)^{1/p^*} \leq C_P\cdot r\cdot \bigg( \fint_B |\nabla u |^p \ \textrm{d}x\bigg)^{1/p}
\end{equation}
where $r$ is the radius of $B$ and $p^* = \frac{np}{n-p}$. The constant $C_P$ depends only on $p$ and $n$.
See e.g. \cite[Lemma 4.9]{EG2015}.
 When a domain possesses some regularity, a similar global inequality holds. 
 However, in contrast to the case above, the constant will depend on the geometry of a particular domain.
 We thus find it convenient to introduce the notion of a domain sequence that supports uniform global Poincar\'e inequality.

\begin{defn} Let $p,q\in[1,\infty]$ and $\Omega_i\subset \R^n$ be bounded domains.
We call $(\Omega_i)_{i\in\mathbb N}$ a $(q,p)$\textit{-Poincar\'e sequence} of domains if
there is a constant $C'_P$ such that the following inequality holds true
\begin{equation}\label{eq:qp-PI}
\bigg(\fint_{\Omega_i} |u(x) - u_{\Omega_i}|^q \ \textrm{d}x\bigg)^{1/q} \leq C'_P\cdot \bigg( \fint_{\Omega_i} |\nabla u |^p \ \textrm{d}x \bigg)^{1/p}
\end{equation} 
for all $u\in W^{1,p}(\Omega_i)$, $i \in \N$.
\end{defn}
In particular, from \eqref{eq:qp-PI}, we have a uniform Sobolev inequality: %
\begin{lem}
Let $(\Omega_i)_{i\in\mathbb N}$ be a $(q,p)$\textit{-Poincar\'e sequence} and suppose that $\sup_{i}|\Omega_i|^{1/q - 1/p} <\infty$.  
Then, there exists a uniform constant $C$ such that 
 \begin{equation}\label{eq:uniformSobolev}
  \|u \|_{L^{q}(\Omega_i)} \leq C \|u \|_{W^{1,p}(\Omega_i)}
 \end{equation}
 for any $u \in W^{1,p}(\Omega_i)$, $i \in \N$ .
\end{lem}
\begin{proof}
Indeed
\begin{align*}
\bigg(\int_{\Omega_i} |u(x)|^q \ \textrm{d}x\bigg)^{1/q}
\leq 
|\Omega_i|^{1/q}\bigg(\int_{\Omega_i} |u(x) - u_{\Omega_i}|^q \ \textrm{d}x\bigg)^{1/q} + |\Omega_i|^{1/q}u_{\Omega_i}\\
\leq |\Omega_i|^{1/q} C'_P\cdot \bigg( \fint_{\Omega_i} |\nabla u |^p \ \textrm{d}x \bigg)^{1/p}
+ |\Omega_i|^{1/q} \bigg( \fint_{\Omega_i} |u | \ \textrm{d}x \bigg)\\
\leq |\Omega_i|^{1/q - 1/p}(C'_P +1) \|u\|_{W^{1,p}(\Omega_i)}.
 \end{align*}

Thus we have \eqref{eq:uniformSobolev},  where we can choose $C = (C'_P +1)\sup_{i}|\Omega_i|^{1/q - 1/p}$.
\end{proof}

The $(q,p)$-Poincar\'e condition is less restrictive then Sobolev extendability of all $\Omega_i$ and is in a good correspondence with usual Sobolev embedding theory. As a simple example of a sequence with this property we consider the following class of domains (this class is also suitable for the following sections).

\begin{defn}[Uniform Lipschitz-graph condition]
  Let $(\Omega_i)_{i \in \N}$ be a sequence of bounded domains that converges to a domain $\Omega$ in the Hausdorff-sense. We say that this sequence satisfies the \emph{uniform Lipschitz-graph condition} if there is a common, locally finite open cover $(U_k)_{k \in \N}$ for all $\partial \Omega_i$ and $\partial \Omega$ as well as a set of corresponding directions $\nu_k \in \mathbb{S}^{n-1}$ such that all sets $\Omega_i \cap U_k$ and $\Omega \cap U_k$ are subgraphs of a Lipschitz-function in direction $\nu_k$ and the Lipschitz-constants of all those functions are uniformly bounded.
\end{defn}

Following the classic proof for the Sobolev-inequality, one can then derive the following:

\begin{cor} \label{cor:LipschitzSobolevInequality}
  Let $(\Omega_i)_{i \in \N}$ be a sequence of bounded domains that converges to a domain $\Omega$ in the Hausdorff-sense. If this sequence satisfies the uniform Lipschitz-graph condition, then it is a $(q,p)$-Poincar\'e sequence for any $p \in [1,\infty]$ and $1 \leq q \leq \frac{np}{n-p}$ for $p<n$ and $q \in [1,\infty)$ for $p\geq n$.
\end{cor}

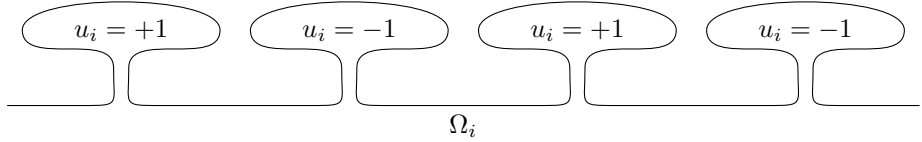
\begin{figure}[h]
 \begin{tikzpicture}
  \draw (0,0) -- (1,0) .. controls +(.5,0) and +(0,-.5) .. (1.4,.5) .. controls +(0,.5) and +(0,-.5) .. (.2,1) .. controls +(0,.5) and +(0,.5) .. (2.8,1) .. controls +(0,-.5) and +(0,.5) .. (1.6,.5) .. controls +(0,-.5,0) and +(-.5,0) .. (2,0) -- (3,0)  -- (4,0) .. controls +(.5,0) and +(0,-.5) .. (4.4,.5) .. controls +(0,.5) and +(0,-.5) .. (3.2,1) .. controls +(0,.5) and +(0,.5) .. (5.8,1) .. controls +(0,-.5) and +(0,.5) .. (4.6,.5) .. controls +(0,-.5,0) and +(-.5,0) .. (5,0) -- (6,0) -- (7,0) .. controls +(.5,0) and +(0,-.5) .. (7.4,.5) .. controls +(0,.5) and +(0,-.5) .. (6.2,1) .. controls +(0,.5) and +(0,.5) .. (8.8,1) .. controls +(0,-.5) and +(0,.5) .. (7.6,.5) .. controls +(0,-.5,0) and +(-.5,0) .. (8,0) -- (9,0)  -- (10,0) .. controls +(.5,0) and +(0,-.5) .. (10.4,.5) .. controls +(0,.5) and +(0,-.5) .. (9.2,1) .. controls +(0,.5) and +(0,.5) .. (11.8,1) .. controls +(0,-.5) and +(0,.5) .. (10.6,.5) .. controls +(0,-.5,0) and +(-.5,0) .. (11,0) -- (12,0);
  \node at (1.5,1) {$u_i=+1$};
  \node at (4.5,1) {$u_i=-1$};
  \node at (7.5,1) {$u_i=+1$};
  \node at (10.5,1) {$u_i=-1$};
  \node[below] at (6,0) {$\Omega_i$};
 \end{tikzpicture}
 \caption{A sketch of domains $\Omega_i$ and functions $u_i$ in Remark \ref{rem:nonUniformLipschitz}.}
\end{figure}

\begin{rem} \label{rem:nonUniformLipschitz}
 Note that the uniformity in the graph-condition is in some sense crucial. Consider a sequence of domains that consists of more and more, increasingly smaller blobs near a main domain but attached only via thin necks. This can be done in a similar way to the examples in \cite[Sec.\ 1.1.4]{M2011}, except that the refinement happens along the sequence of the domains instead of along the boundary of a fixed domain. This way, the values of $u_i$ on those blobs can differ strongly from that on the main domain without incurring much of a penalty in the main domain. Nevertheless, each domain itself is Lipschitz.
 
 Then if the values on those blobs alternate in sign, the weak limit of $\widetilde{u}_i$ has to be zero, but $\| u_i\|_{L^p(\Omega_i)}$ can be chosen as a fixed nonzero number, which shows that there is no uniform embedding constant and also for the next section that there can be no strong convergence and thus no compact subsequence. In higher dimensions this can even be done in such a way that $\partial \Omega_i$ has bounded mass.

 Note also that by careful choice of parameters in this construction, one can find examples of sequences of domains which are only $(q,p)$-Poincar\'e for $1 \leq q \leq q_0$ for some given $q_0 < \frac{np}{n-p}$.
\end{rem}

%To formulate a compactness result we will assume that all $\Omega_i$ and $\Omega$ are in some set $U \subset\mathbb R^n$ of finite measure 
%(e.g. it holds if $\Omega$ is bounded). 

\subsection{Strong compactness}
Let $U\subset \mathbb R^n$ bounded and $\Omega\subset U$.
We will use the following properties of the Hausdorff convergence. 
For $\alpha>0$ define
\[
\Omega_{\alpha+} = \{x\in U : \dist(x, \Omega) <\alpha \},  \quad \Omega_{\alpha-} = \{x\in \Omega : \dist(x, U\setminus \Omega) <\alpha \}.
\]

 \begin{prop}\label{prop:alpha+-}
 Let $(\Omega_i)_{i\in\mathbb N}$ be a sequence of domains that converge to a domain $\Omega$ in the Hausdorff-sense. 
 Then for any $\alpha>0$ there is $N\in\N$ such that for all $i> N$ 
  \[
   \Omega_i \subset \Omega_{\alpha+}  \quad \text{and} \quad \Omega_{\alpha-} \subset \Omega_i.
 \]
 \end{prop}
  
\begin{lem}\label{lemma:Minkovskii-dim}
Let $\Omega\subset\R^n$ be a bounded domain with 
$\partial \Omega$ of box-counting dimension less than $n$. 
Then for any $\varepsilon>0$ there is $\alpha>0$ so that $|\Omega_{\alpha+} \setminus  \Omega_{\alpha-}| < \varepsilon$.
\end{lem}

\begin{proof}
 We cover $\partial \Omega$ with $N(\alpha)$ balls $\{B_\alpha(x_i)\}_i$ of radius $\alpha$. Then, as any point in $\Omega_{\alpha+} \setminus \Omega_{\alpha-}$ is at most of distance $a$ to $\partial \Omega$, the set is covered by $\{B_{2\alpha}(x_i)\}_i$. Now
 \begin{align*}
  |\Omega_{\alpha+} \setminus \Omega_{\alpha-}| \leq \sum_{i=1}^{N(\alpha)} |B_{2\alpha}(x_i)| \leq N(\alpha) 2^n \omega_n \alpha^n \to 0
 \end{align*}
 for $\alpha \to 0$ as per the definition of the box-counting dimension.
\end{proof}

Note that this, in particular, is true for Lipschitz domains.

\begin{thm}[Rellich-Kondrachov type theorem]\label{theorem:RK2} 
Let $(\Omega_i)_{i\in\mathbb N}$ be a $(q,p)$-Poincar\'e sequence of domains and $(\Omega_i)_{i\in\mathbb N}$ converge to a domain $\Omega$ in the Hausdorff-sense.
Suppose that $\partial \Omega$ is of box-counting dimension less than $n$.
Then for any sequence $u_i \in W^{1,p}(\Omega_i)$ such that 
$\sup_i \|u_i\|_{W^{1,p}(\Omega_i)} <\infty$
there exists a subsequence (not relabeled) and a limit $u\in L^q(\Omega)$ such that $\widetilde{u}_i \to \widetilde{u}$ in $L^{\alpha}(\R^n)$
for every $1\leq\alpha < q$.
\end{thm}
%
% \begin{rem}
%     Note that here we do not state explicitly the bounds for $q$, as these bounds contains in the Poincar\'e inequality for $\Omega_i$. For example, if all $\partial\Omega_i$ are Lipschitz and $p < n$, the Poincar\'e inequality holds with $q=\frac{np}{n-p}$. In such case, the statement of the above theorem holds for $1 \leq \alpha < \frac{np}{n-p}$.
% \end{rem}

We use the following
\begin{lem}[{\cite[Lemma 8.2]{HK2000}}]\label{lemma:8.2}
Let $Y$ be a set equipped with a finite measure $\nu$. 
Let $1 < q < \infty$ and $(v_i)_{i \in \N}$ be a bounded sequence in $L^q(Y)$. 
If $v_i$ converges in measure to $v \in L_q(Y)$, then $v_i$ converges to $v$ in the norm of $L^{\alpha}(Y)$ for every $1\leq \alpha < q$.
\end{lem}

\begin{proof}[Proof of Rellich-Kondrachov theorem]
By the uniform Sobolev inequality, the sequence $\tilde{u}_i$ is bounded in $L^{q}(\R^n)$. 
Then, there exists a subsequence (not relabeled) that converges weakly to some $u\in L^q(\mathbb R^n)$.
Moreover,  by the same arguments as before, $u$ is supported on $\Omega$.
We can assume that all $\Omega_i$ and $\Omega$ are in some set $U \subset\mathbb R^n$ of finite measure. 
So, in order to apply Lemma \ref{lemma:8.2}, we have to prove that  $\tilde{u}_i$ converges to $u$ in measure, i.e. for any $\varepsilon>0$
\[
\mathcal L(\varepsilon, i, U) = |\{x\in U : |\tilde u_i(x) - u(x)| >\varepsilon \} | \to 0 \quad \text{ as } i\to\infty.
\]

Now fix $\varepsilon>0$. We have
\[
\mathcal L(\varepsilon, i, U) =\mathcal L(\varepsilon, i, U\setminus \Omega_{\alpha+}) 
+ \mathcal L(\varepsilon, i, \Omega_{\alpha+} \setminus  \Omega_{\alpha-}) 
+ \mathcal L(\varepsilon, i, \Omega_{\alpha-}).
\]
Then by Lemma \ref{lemma:Minkovskii-dim} for any $\varepsilon'>0$ we find $\alpha>0$ so that
$ \mathcal L(\varepsilon, i, \Omega_{\alpha+} \setminus  \Omega_{\alpha-}) < \varepsilon'$ (for any $i\in\N$). 
 In turn by Proposition \ref{prop:alpha+-} there is such an $N\in\N$ that for all $i>N$ the functions $\tilde{u}_i$ are supported in $\Omega_{\alpha+}$,
which implies $ \mathcal L(\varepsilon, i, U\setminus \Omega_{\alpha+}) = 0$.

Thus, the third term remains to be estimated. 
Again, due to Proposition \ref{prop:alpha+-}, we can assume that $\Omega_i\supset \Omega_{{\alpha}/{2}-}$ for $i>N$.
For $r<\alpha/2$ and $x\in\Omega_{\alpha-}$ we have $B(x,r)\subset\Omega_{\alpha/2-}$.
Therefore, we can consider  
\[
{(u_i)}_{B(x,r)} = \fint_{B(x,r)} u_i(x) \ \d x \quad \text{ and }\quad {u}_{B(x,r)} = \fint_{B(x,r)} u(x) \ \d x.
\]

Then  $\mathcal L(\varepsilon, i, \Omega_{\alpha-}) \leq |A_1(i,r)| + |A_2(i,r)| + |A_3(r)|$,
where 
\begin{align*}
A_1(i,r) &= \{x \in \Omega_{\alpha-}  :  |u_i(x) - {(u_i)}_{B(x,r)}| \geq \varepsilon/3\};\\
A_2(i,r) &= \{x \in \Omega_{\alpha-}  :   |(u_i)_{B(x,r)} -  u_{B(x,r)}| \geq \varepsilon/3\};\\
A_3(r) &= \{x \in \Omega_{\alpha-}  :  |u(x) -  u_{B(x,r)} | \geq \varepsilon/3\}.
\end{align*}

1) Let $x\in\Omega_{\alpha-}$ be a common Lebesgue point for all $u_i$. 
%Denote $B_i(x) = B(x,2^{-i}r)$. 
Denote $r_j = 2^{-j}r$. So we have $\lim_{j\to\infty} (u_i)_{B(x, r_j)} = u_i(x)$.
Using the H\"older inequality and \eqref{eq:PI}, we derive an estimate via the maximal function
\begin{align*}
&\phantom{{}={}}\|u_i(x) - (u_i)_{B(x,r)}|
\leq \sum_{j=0}^{\infty} |(u_i)_{B(x,r_j)} - (u_i)_{B(x,r_{j+1})}|\\
&\leq \sum_{j=0}^{\infty} \fint_{B(x,r_{j+1})}   |u_i(y) - (u_i)_{B(x,r_j)}| \ \textrm{d}y
\leq 2^{n}\sum_{j=0}^{\infty} \fint_{B(x,r_j)}   |u_i(y) - (u_i)_{B(y, r_j)}| \ \textrm{d}y\\
&\leq 2^{n}\sum_{j=0}^{\infty}  \bigg( \fint_{B(x,r_j)}   |u_i(y) - (u_i)_{B(y, r_j)}|^{p^*} \ \textrm{d}y\bigg)^{1/p^*}\\
&\leq 2^{n}\sum_{j=0}^{\infty}  C_P 2^{-j}r \bigg( \fint_{B(x, r_j)}   |\nabla u_i(y)|^{p} \ \textrm{d}y\bigg)^{1/p}
\leq Cr\big( M |\nabla u_i|^{p}(x)\big)^{1/p}.
\end{align*}
From the boundedness of the maximal operator, we conclude 
\begin{align*}
|\{x \in \Omega_{\alpha-}  :  |u_i(x) - {(u_i)}_{B(x,r)}| \geq\varepsilon/3\}|
\leq |\{x \in \Omega_{\alpha-}  : M |\nabla u_i|^{p}(x) \geq (1/C)^pr^{-1} (\varepsilon/3)^p \}|\\
\leq\frac{C_1 r^{p}}{(1/C)^p (\varepsilon/3)^p} \| |\nabla u_i|^{p} \|_{L^1(\Omega_{\alpha-})} \to 0 \quad \text{ as } r\to 0
\end{align*}
uniformly by $i$.

2) By previous $\tilde{u}_i$ converges weakly to $u$ in $L^q(\mathbb R^n)$. In particular, $(u_i)_{B(x,r)} \to  u_{B(x,r)}$ for any $x$ and $r$. 
Thus, for any $r>0$ we have $|A_2(i,r)| \to 0$ as $i\to\infty$.

3) We have that $u_i - (u_i)_{B(\cdot,r)}$ converges weakly to $u - u_{B(\cdot,r)}$ in $L^p(\Omega_{\alpha-})$.
Then, considering $u_i - (u_i)_{B(\cdot,r)}$  as a sequence of bounded linear operators on $L^{q'}(\Omega_{\alpha-})$ and applying the Banach-Steinhaus theorem
we obtain
\[
\|u - u_{B(\cdot,r)} \|_{L^q(\Omega_{\alpha-})} \leq \sup_i \|u_i - (u_i)_{B(\cdot,r)} \|_{L^q(\Omega_{\alpha-})} .
\]
Then, using the same estimate as in 1), we conclude $\|u - u_{B(\cdot,r)} \|_{L^q(\Omega_{\alpha-})} \to 0$ when $r\to 0$.

%Finally, applying the same methods as in the proof of \cite[Theorem 8.1]{HK2000} (we have only one domain now), we derive $\mathcal L_3(\varepsilon, n, \Omega_{\alpha-}) \to 0$ as $n\to\infty$.
The theorem follows.
\end{proof}

\section{Boundary values}

In many applications, an important but often neglected part of the problem consists of taking the right boundary conditions. For fixed domains these boundary related difficulties are often not immediately evident, as the arguments used there are quite standard and one simply refers to standard textbook-results, guaranteeing e.g.\ the existence and continuity of the trace operator for Lipschitz-domains.

In variable domains in contrast, the situation gets much more interesting. As the domain changes, so does the boundary. Thus it is a-priori not even clear what continuity of the trace operator even means. As a consequence, we postpone a detailed general study of boundary conditions to a future work. However, as the topic is so important for applications, in this section, we will give some general useful results, as well as results on some special cases that already suffice for many problems.

\subsection{Trace operators for parametrized boundaries}

As will be seen when discussing ALE-convergence, perhaps the easiest situation for changing domains is when there is some sort of natural parametrization of the domain by a map from some given reference domain. As was also mentioned this is not the case for most problems. 

That being said, while a reference domain is a rare occurrence, there are many problems which involve a reference boundary, i.e.\ where at least locally the changing boundary has a natural reference configuration. A classic example for this is fluid-structure interaction, where the boundary of the changing fluid domain is given by a solid, which in turn is modeled by a deformation map from a given reference configuration.

\begin{thm}[Boundary parametrization] \label{thm:boundaryParam}
 Let $(\Omega_i)_{i\in\mathbb N}$ be a sequence of domains that converges to a domain $\Omega$ in the Hausdorff-sense and $(\Omega_i)_{i \in \N}$ satisfy the uniform Lipschitz-graph condition. Assume that there is an $n-1$ dimensional manifold $\omega$ and a sequence of diffeomorphisms $\eta_i : \omega \to \partial \Omega_i$ and similarly $\eta: \omega \to \partial \Omega$. Assume also that $\eta_i \to \eta$ in $W^{1,\infty}(\omega;\R^n)$.
 Then for any converging sequence $(u_i)_{i \in \N}$, $u_i \in W^{k,p}(\Omega_i)$, $u_i \to u$ in the zero-extension sense, $u \in W^{1,p}(\Omega)$, we have
 \begin{align*}
  u_i \circ \eta_i \to u \circ \eta \text{ in } L^p(\omega;\R^n).
 \end{align*}
 where $u \in L^p(\partial \Omega_i)$ is understood in the trace sense.
\end{thm}

\begin{proof} 
 First of all note that if $\Omega_i = \Omega$ is a fixed Lipschitz domain and only the $\eta_i$ vary, then this is a direct consequence of the trace theorem and the change-of-variables formula for Lipschitz functions. Using this together with a cutoff, we can restrict ourselves to a local graph setting, where furthermore the boundary parametrization is of the form $\eta_i: \omega\subset \R^{n-1} \to \R^n; x\mapsto (x,h_i(x))$. 
Furthermore by Lemma \ref{lem:smoothApprox} and a density argument, we can assume $u$ and the $u_i$ to be smooth.
 
 Next, looking more closely at the standard proof of the trace theorem, we can estimate
 \begin{align*}
  |u_i \circ \eta_i (x) - u \circ \eta(x)| &= \left|\int_{-\infty}^{h_i(x)} \partial_n u_i(x,s) ds - \int_{-\infty}^{h(x)} \partial_n u(x,s) ds \right|\\
  &\leq \int_{-\infty}^{\infty} | (\partial_n u_i)^\sim(x,s) - (\partial_n u)^\sim(x,s)| ds.
 \end{align*}
 Integrating this over $x\in \R^{n-1}$ shows the result for $p=1$. Applying this again for $v_i := |u_i|^p$ and estimating the last term again using the chain rule and H\"older's inequality proves it for $p>1$.
\end{proof}

Note that this also immediately implies the following corollary.
\begin{cor}
 Let $(\Omega_i)_{i\in\mathbb N}$ be a sequence of domains that converges to a domain $\Omega$ in the Hausdorff-sense and $(\Omega_i)_{i \in \N}$ satisfy the uniform Lipschitz-graph condition. Then for any $i$ and $u \in W^{1,p}(\Omega_i)$ we have
 \begin{align*}
  \| u \|_{L^p(\partial \Omega_i)} \leq C \| u\|_{W^{1,p}(\Omega_i)}
 \end{align*}
 where the constant $C$ does not depend on $i$.
\end{cor}

\subsection{Zero boundary values}

The results from the previous subsection in particular imply that under the above conditions, a limit of $W_0^{k,p}(\Omega_i)$-functions will be in $W_0^{k,p}(\Omega)$. However in that case, there is no need for a parametrization as the zero-extension is in fact the natural extension. Specifically for any $u \in C_c^\infty(\Omega)$, we have that $u^\sim \in C_c^\infty(\R^n)$. Furthermore, this is trivially continuous in the $W^{1,p}$-norm and thus $\cdot^\sim : W^{1,p}_0(\Omega) \to W^{1,p}(\R^n)$ is an isometric embedding. 

\begin{thm}
 Let $(\Omega_i)_{i \in \N}$ be a sequence of bounded domains that converges to a domain $\Omega$ in the Hausdorff-sense. Then for any sequence $u_i \in W_0^{1,p}(\Omega_i)$ such that $u_i \to u \in W^{1,p}(\Omega)$ in the zero extension sense we have that $u \in W^{1,p}_0(\Omega)$.
\end{thm}

\begin{proof}
By Lemma \ref{lem:smoothApprox}, and a density argument, we can again assume $u$ and the $u_i$ to be smooth.

For any $\varepsilon$ we can now pick a smooth cutoff-function $\psi_\varepsilon:\R^n \to [0,1]$ such that $\psi(x) = 0$ in an $\varepsilon$-neighborhood of $\partial \Omega$, and $\psi(x) = 1$ for all $x\in \Omega$ such that $\operatorname{dist}(x,\partial \Omega) > 2\varepsilon$. Additionally we can assume $|\nabla \psi_\varepsilon| \leq \frac{C}{\varepsilon}$. Then for all large enough $i$, we have that $\psi_\varepsilon u_i \in C_c(\Omega_i)$. Additionally we have $\psi_\varepsilon u_i \to \psi_\varepsilon u$ for $i\to \infty$ and $\|\psi_\varepsilon u_i - u_i\|_{W^{1,p}(\Omega_i)} \leq C(\varepsilon) \to 0$ for $\varepsilon \to 0$. The statement then follows from a diagonal sequence.
\end{proof}

Finally note that there is the following related approximation result:
\begin{thm}
 Let $(\Omega_i)_{i \in \N}$ be a sequence of domains converging to a bounded domain $\Omega$ in the Hausdorff-sense. Then for any $u \in W^{1,p}_0(\Omega)$ there exists a sequence $u_i \in W^{1,p}_0(\Omega_i)$ such that $u_i \to u$ in the zero-extension sense.
\end{thm} 

\begin{proof}
 Per definition of $W^{1,p}_0(\Omega)$ there exists a sequence $(v_j)_j \in C_c^\infty(\Omega)$ such that $v_j\to u$ in $W^{1,p}(\Omega)$. But then for every $j$, we have $\operatorname{dist}(\supp v_j, \partial \Omega ) > 0$. Thus there exists an $i(j)$ such that $\supp v_j \subset \Omega_i$ for all $i\geq i(j)$, which means that $(v_j)^\sim \in W^{1,p}_0(\Omega_i)$. From these one can then pick the required approximating sequence.
\end{proof}

Note that even the boundedness of $\Omega$ could be relaxed by a localization argument.

\section{Examples}

\subsection{A shape optimization example} \label{subsec:shapeOpt}

\begin{figure}[h]
 \begin{tikzpicture}
  \draw[->] (0,0) -- (0,2.5);
  \draw[->] (0,0) -- (5,0);
  \draw (0,1.5) .. controls +(1,.3) and +(-1,-.3) .. (2,1.4) .. controls +(1,.3) and +(-1,-.3) .. (4,1.5);
  \node[above] at (2,1.4) {$r$};
  \draw[dashed] (0,0) ellipse (.5 and 1.5);
  \draw[dashed] (4,0) ellipse (.5 and 1.5);
  \draw[->] (.2,.5) -- (.7,.5);
  \draw[->] (4.2,.5) -- (4.7,.5);
  \node[right] at (0.7,.5) {$u_{\text{in}}$};
  \node[right] at (4.7,.5) {$u_{\text{out}}$};
  \draw[dashed] (0,-1.5) .. controls +(1,-.3) and +(-1,.3) .. (2,-1.4) .. controls +(1,-.3) and +(-1,.3) .. (4,-1.5);
  \node at (2,-.6) {$\Omega_r$};

  \draw (6,1.5) .. controls +(1,.3) and +(-1,-.3) .. (8,1.4) .. controls +(1,.3) and +(-1,-.3) .. (10,1.5);
  \draw (6,1.0) .. controls +(1,.2) and +(-1,-.1) .. (8,1.0) .. controls +(1,.1) and +(-1,-.3) .. (10,1.0);
  \draw[dashed] (6,1.5) -- (6,1);
  \draw[dashed] (10,1.5) -- (10,1);
  \node[right] at (10,1.5) {$h_i^+$};
  \node[right] at (10,1.0) {$h_i^-$};
  \node at (8.5,1.2) {$\Omega_i$};

  \draw (6,-.2) .. controls +(1,.3) and +(-.5,-.15) .. (8,-.1) .. controls +(.5,.15) and +(-.5,.05) .. (10,-.3);
  \draw (6,-.5) .. controls +(1,.2) and +(-1,-.3) .. (8,-.5) .. controls +(1,.3) and +(-.5,.05) .. (10,-.3);
  \draw[dashed] (6,-.2) -- (6,-.5);
  \node[right] at (10,-.3) {$h^+(x_0)=h^-(x_0)$};
  \node at (7.5,-.35) {$\Omega$};
 \end{tikzpicture}
 \caption{The domains in Subsections \ref{subsec:shapeOpt} and \ref{subsec:nonGraph}.}
\end{figure}
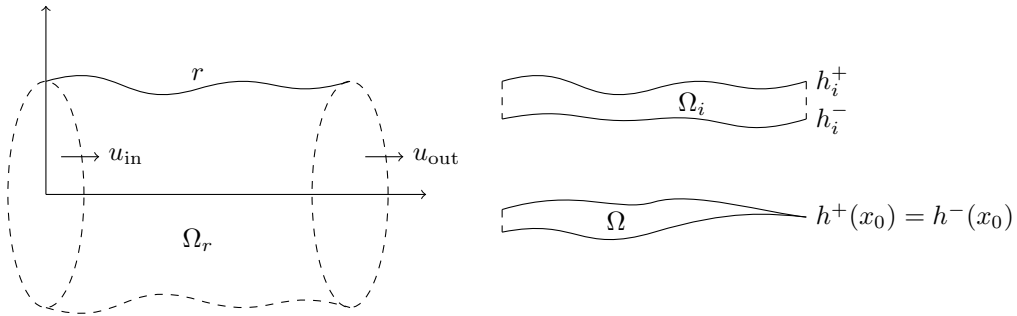

Problems involving changing domains often occur in the field of fluid-structure interaction where the equations of motions are solved on a domain whose shape is not fixed but in turn determined by the deformation of a solid whose determination is part of the problem. Consider the following toy-example\footnote{Note that this example is simple enough to be solved by more classical approaches as well. However it illustrates the possibilities of the framework.} inspired by results in \cite{muhaExistenceWeakSolution2013} and \cite{bucur2024free} (compare also \cite{bucur2005variational}):

For a regular enough function $r: [0,1] \to (0,\infty)$ with $r(0)$ and $r(1)$ fixed consider the rotational domain of radius given by $r$, i.e.\
\[ \Omega_r := \{(x_1,x_2,x_3) \in \R^3: x_1\in (0,1), x_2^2+x_3^2 < r(x_1)^2 \}. \]

On this domain one can consider a fluid flow problem for the stationary Stokes-equation
\begin{equation} \label{eqStokes}
 \left\{ \begin{aligned} \Delta u -\nabla p &= 0 &\text{ in }&\Omega_r \\ \nabla \cdot u &= 0 &\text{ in }& \Omega_r \\ u &= u_{\text{in}} &\text{ in }&\partial \Omega_r\cap \{x_1= 0\}\\ u &= u_{\text{out}} &\text{ in }&\partial \Omega_r\cap \{x_1= 1\} \\ u &= 0 &\text{ in }&\partial \Omega_r\cap \{x_1 \in (0,1)\} \end{aligned} \right.
\end{equation}

This is a simple model of the flow of an incompressible fluid through a pipe shaped domain for a given in- and outflow $u_{\text{in}}$ and $u_{\text{out}}$. As long as total in and outflow are the same, existence and uniqueness of weak solutions to this problem can be shown using standard methods.

A function $u \in W^{1,2}(\Omega_r)$ is weak solution to this PDE if it satisfies the given boundary conditions as well as
\begin{equation} \label{eqWeakStokes}
 \int_{\Omega_r} \nabla u : \nabla \varphi \,dx   = 0  \text{ for all } \varphi \in C_0^\infty(\Omega_r;\R^3), \nabla \cdot \varphi = 0.
\end{equation}

We now want to treat this as a problem in $r$. Assume we have a certain cost functional $\mathcal{P}(r)$ (e.g. a price of construction) and note that from physical considerations the fluid drag is given by $\mathcal{D_r}(u_r) := \int_{\Omega_r} |\nabla u_r|^2$ where $u_r$ is the unique solution to \eqref{eqWeakStokes}.

We can now ask all kind of optimization questions, e.g.\ what $r$ has the cheapest $\mathcal{P}(r)$ for a given maximal drag $\mathcal{D}(r) \leq d_0$, or the least drag for a given price, or we can even try to minimize an arbitrary combination of the two.

The classic procedure for showing existence of minimizers of such problems is the direct method. Start with a minimizing sequence $(r_i,u_i)_{i \in \N}$, use compactness in order to find a converging subsequence and a limit and then finally proof that the functionals are lower-semicontinuous in order to show that the limit is a minimizer. 

The interesting situation here is that the domain of $u_i$ itself depends on $r_i$. So while compactness of $r$ is standard (a typical assumption here would be e.g. that $\mathcal{P}$ is of such a form that any set $\{\mathcal{P}(r)<r_0\}$ embeds compactly into $C^{1,\alpha}$ for some $\alpha >0$), compactness of $u_i$ involves changing domains.

Assume that (after taking a subsequence) $r_i \to r_0$ in $C^{1,\alpha}((0,1))$. Then $\Omega_{r_i}$ converges to $\Omega_{r_0}$ in the Hausdorff-sense. In addition to that the domains $(\Omega_{r_i})_{i \in \N}$ are uniformly Lipschitz and thus by Corollary \ref{cor:LipschitzSobolevInequality} form a $(6,2)$-Poincar\'e sequence of domains. Then by Theorem \ref{BA-internal} there exists another subsequence such that finally $u_i \rightharpoonup u_0$ in the $W^{1,2}$-zero extension sense. A straightforward calculation shows that $u_i$ solves \eqref{eqWeakStokes} for $r=r_0$ and that $\mathcal{D}_{r_0}(u_0) \leq \liminf_{i\to \infty} \mathcal{D}_{r_i}(u_i)$. If $\mathcal{P}$ has a similar lower-semicontinuity, then $(r_0,u_0)$ will be the required minimizer.

While the equation in this case was linear, it is not hard to see that similar things can be done with a non-linear PDE, e.g.\ by including transport term $u_i \cdot \nabla u_i$ to obtain the steady Navier-Stokes equation. In this case, the definition of a weak solution looks like,
\begin{equation*}
 \int_{\Omega_r} \nabla u : \nabla \varphi + (u \otimes u) : \nabla \varphi \,dx  = 0  \text{ for all } \varphi \in C_0^\infty(\Omega_r;\R^3), \nabla \cdot \varphi = 0.
\end{equation*}
This additional term is nonlinear. However for convergence of the equation we can now rely on Theorem \ref{theorem:RK2} and proceed in the same way.

\subsection{\texorpdfstring{$\Gamma$}{Γ}-convergence}

A common strategy in the calculus of variations is to find minimizers to one problem by considering minimizers to a sequence of different, related problems instead. 
These problems are picked in such a way that proving the existence of a minimizer is easier, yet still the sequence of those minimizers then converges to a minimizer of the limit problem. For the latter, one often relies on a concept called $\Gamma$-convergence. Formally this can be summarized in a simple theorem. (See e.g.\ \cite{braides2002gamma} for a more complete introduction to the topic)

\begin{thm}
 Let $(X_i)_{i \in \N}$ be a sequence of topological spaces and assume that there is a notion of convergence $v_i \to v$ for $v_i \in X_i$ and $v$ in a limit space $X$. Consider a sequence of functionals $\mathcal{F}_i: X_i \to \R$ and a limit functional $\mathcal{F}: X \to \R$. Assume that all $\mathcal{F}_i$ have a minimizer $u_i \in X_i$ and that we have the following:
 \begin{enumerate}
  \item $\liminf$-inequality: For any sequence $v_i \in X_i$ with $v_i \to v \in X$, we have
  \[\mathcal{F}(v) \leq \liminf_{i\to \infty} \mathcal{F}_i(v_i)\]
  \item $\limsup$-inequality: For any $v \in X$, there exists a sequence $v_i \in X_i$ such that 
  \[\mathcal{F}(v) \geq \limsup_{i\to \infty} \mathcal{F}_i(v_i)\]
  \item Compactness: For any sequence $v_i \in X_i$ such that $\mathcal{F}_i(v_i)$ is bounded, there exists a converging subsequence such that $v_i \to v \in X$.
 \end{enumerate}
 Then $\mathcal{F}$ has a minimizer.
\end{thm}

The proof consists of applying the direct method to show that the sequence of minimizers has a limit $u\in X$ and then using the $\limsup$-inequality to show that the result has to be the minimizer, as any improved competitor would result in a contradiction.

What interests us here is how well this type of argument fits our theory. Consider a sequence of domains $\Omega_i \to \Omega$ and functionals $\mathcal{F}_i: W^{k,p}(\Omega_i) \to \R$ with a limit $\mathcal{F}:W^{k,p}(\Omega) \to \R$. Then it is quite natural to consider some type of zero-extension convergence for the problem.

In particular the compactness theorems derived in Section \ref{sec:compactness} can be applied to provide condition (iii) in the preceding theorem. For deriving the $\liminf$ condition, the zero-extension can be of direct use, as it allows the use of known standard lower-semicontinuity results for fixed domains.

Finally the key step in the $\limsup$ inequality is the existence of a recovery sequence. In other words, we need to ask if any function in $W^{k,p}(\Omega)$ can be approximated by a matching sequence $v_i \in W^{k,p}(\Omega_i)$. This is for example the case if $\Omega$ is sufficiently regular enough to have an extension operator, as we then can take $v_i := \mathrm{Ext}(u)|_{\Omega_i}$ (compare also Theorem \ref{zeroToEconvergence}). 

\subsection{A non-graph case} \label{subsec:nonGraph}

While most of the previous results e.g.\ those for compactness focused on a local graph setting, it should be emphasized that the underlying notion of convergence is more general.

Consider e.g.\ a domain that is (locally) given as the area between two graphs
\begin{align*}
 \Omega_i := \{(x,x_n) \in Q \times \R: h^-_i(x) < x_n < h^+_i(x) \}
\end{align*}
for some fixed domain $Q \subset \R^{n-1}$ and sequences $(h^-_i)_i$, $(h^+_i)_i$ with respective limits $h^-$ and $h^+$. As long as $h^-(x) << h^+(x)$, this is essentially just the standard ``local graph''-situation.

However if there is a point $x_0 \in Q$ such that $h^-(x_0)=h^+(x_0)$, then the picture changes drastically. While the domain still has a nice description, it looses a lot of regularity. Already for $n=2$ and $h^-(x)<h^+(x)$ for all $x \neq x_0$, the resulting domain will potentially end up with a cusp and regularity of $h^-$ and $h^+$ cannot improve the situation. In fact, as soon as both are differentiable in $x_0$, the cusp becomes unavoidable. In higher dimensions this only gets worse. See e.g.\ \cite{kampschulteGlobalWeakSolutions2024} for a practical situation in which such a geometry occurs.

Despite all of those problems, it turns out that the framework presented here is able to deal with this situation with little to no issues. The only real danger is that of concentration, as locally the underlying volume gets arbitrarily small. However that problem is neatly dealt with by the fact that convergence in the zero-extension also implies convergence the norm.

It is only in the reverse that issues begin to appear. It is well known that for such cuspidal limit domains there is no longer any classic Sobolev-type inequality. Intuitively as the domain becomes thinner and thinner, it becomes harder to control function values through a larger number of their neighbors. This also has direct implications on the Rellich-Kondrachov type embedding theorems.
%
% \noteMalte[inline]{This would be a useful point to say something about $(p,q)$-Poincare. If I am not completely mistaken, in this example, $q= \frac{dp}{d-p}$ does not work anymore, but $q=\frac{(d-1)p}{(d-q)-p}$ still should, as this behaves more like the underlying domain $Q$.}
%
% \noteSasha[inline]{I don't sure that this fits good in our situation. There exist some versions of the Poincare inequality on $n-1$-dimensional spheres, but it looks far from what we need. Maybe instead we can use that all $\Omega_i$ are Poincare domain, but the limit domain $\Omega$ don't.}

\subsection{Cases without compactness}

Finally let us note that we are even able to talk about convergence in cases that are normally excluded as they lack any sort of compactness. Consider $\Omega_i := (0,1) \cup (1+\frac{1}{i},2) \subset \R$. Then $\Omega_i$ converges to $\Omega = (0,2)$ in measure and there is a clear notion of convergence in the zero-extension sense. 

But there is a lack of Hausdorff-convergence as the middle interval $[1,1+\frac{1}{i}] \subset \R \setminus \Omega_i$ is nowhere near $\R \setminus \Omega$. As an interesting consequence of this, there is not even a Banach-Alaoglu type theorem. Clearly the sequence
\begin{align*}
 u_i : \Omega_i \to \R; x\mapsto \begin{cases} 1  & \text{ for } x \in (0,1) \\ 0 & \text{ for } x \in (1+\frac{1}{i},2) \end{cases}
\end{align*}
is uniformly bounded in any $W^{k,p}$. But any weak limit would have to correspond to the pointwise a.e.\ limit $\chi_{(0,1)}$, which is not in $W^{k,p}(\Omega)$ for any $k \geq 1$ and $p > 1$.

\section{Comparison with other notions of convergence}

In this section we will compare the Zero-Extension convergence to other notions of convergence that have been proposed and used in the literature previously. Before we can start comparing them, we will give a short overview over the definitions.

\subsection{Arbitrary Lagrangian-Eulerian (ALE)-maps}

Perhaps the most common approach to deal with changing domains is the so called ``Arbitrary Lagrangian-Eulerian'' (commonly shortened to ALE) approach (see for example \cite{alphonseAbstractFrameworkParabolic2015}). In this approach, all domains of a family $(\Omega_i)_{i\in I}$ under study\footnote{Here $I$ denotes a general set of indices, not necessarily $I=\N$, cmp.\ e.g.\ the example in subsection \ref{subsec:shapeOpt}, where it would be a class of functions $r$ with a certain regularity.} are identified with a single fixed reference domain $\Omega_0$ by a corresponding family of diffeomorphisms $(\Phi_i)_i$, with $\Phi_i: \Omega_0 \to \Omega_i$. Then instead of considering a function $u:\Omega_i \to \R^n$, one can always consider the pullback $u \circ \Phi_i : \Omega_0 \to \Omega_i$.

As long as the $\Phi_i$ are regular enough, this pullback gives an isomorphism between the vector spaces
 $W^{k,p}(\Omega_i)$ and $W^{k,p}(\Omega_0)$. Thus all questions of convergence can be reduced to considering the fixed space $W^{k,p}(\Omega_0)$.

\begin{defn}[ALE-convergence]
 Under the above assumptions we say that $u_i \in W^{k,p}(\Omega_i)$ ALE-converges to $u \in W^{k,p}(\Omega)$ if and only if
 $u_i \circ \Phi_i$ converges to $u$ in $W^{k,p}(\Omega)$.
% \begin{align*}
% u_i \circ \Phi_i \to u \quad \text{ in } W^{k,p}(\Omega).
%\end{align*}
\end{defn}

This method has several advantages in applications. For example, when studying a PDE on a varying domain, it allows to also pull the equation itself back to the reference set and solve it there, at the cost of additional $\Phi$-dependent coefficients. Similarly, it is a common method in the numerics of finite elements, where it is already common to relate the individual irregular cells of a given mesh to a regular reference polytope and thus different meshes with the same topology can easily be compared.

At the same time, these diffeomorphisms also form a big downside of the approach, as there is usually no useful canonical family of diffeomorphisms for any given problem.\footnote{In some sense, this is what the ``arbitrary'' in the name indicates. In many problems coming from continuum mechanics e.g. fluid flow, there is indeed a formal, canonical way to do so, by fixing as $\Omega_0$ the configuration at a certain time $t_0$ and then following the particles around, i.e. identifying each point $y$ in a later configuration $\Omega_t$ by considering the point in $\Omega_0$ where the particles which are now at $y$ originally started. This is commonly called the Lagrangian representation (compared to the Eulerian, which corresponds to what we are doing). However for practical problems there is generally no way of proving that this approach results in a well-defined family of diffeomorphisms. For example, for the Navier-Stokes equation in 3D, obtaining sufficient regularity would require solving the famous Millennium problem.} One has thus usually has to construct a new family of diffeomorphisms for every new problem. This is not always trivial, since convergence of domains has to result in a good enough convergence of the corresponding diffeomorphisms. This has to hold to the desired order of derivatives, as the derivatives of $\Phi$ all occur in those of $u\circ \Phi$ via the chain rule.

All of this makes this notion of convergence rather rigid. In particular, it cannot handle changes in topology, nor does it allow for much freedom of the boundary regularity, as $\partial \Omega_0$ and $\partial \Omega_i$ will always be related by a diffeomorphism as well.

It is thus not surprising that the resulting notion of convergence is only a special case of the one we study. However instead of proving this directly, we note that the ALE-approach is directly seen to be a special case of the so-called E-convergence.

\subsection{E-convergence}

The notion of E-convergence originally goes back to Vainikko \cite{vainikkoRegularConvergenceOperators1981} and has mainly been used to study convergence of eigenvalues and of operators. For example, see the recent articles \cite{FL22, F21}. Its main concept is that of a connecting system.

Assume that for a sequence of domains $(\Omega_i)_{i \in \mathbb{N}}$ and a sequence of corresponding Sobolev spaces $(W^{1,p}(\Omega_i))_{i \in \N}$ there exists a sequence of linear operators (called a \emph{connecting system})
$$
    E_i: W^{1,p}(\Omega) \to W^{1,p}(\Omega_i)
$$
such that
$$
    \|E_i u\|_{W^{1,p}(\Omega_i)} \to \|u\|_{W^{1,p}(\Omega)} \quad \text{for all } u \in W^{1,p}(\Omega).
$$

\begin{defn}
    The sequence $(u_i)_{i \in \N}$, $u_i \in W^{1,p}(\Omega_i)$, $E$-converges to $u \in W^{1,p}(\Omega)$ if 
    $$
        \|u_i - E_i u\|_{W^{1,p}(\Omega_i)} \to 0 \quad \text{for } i \to \infty.
    $$
\end{defn}

\begin{cor}[ALE implies E-convergence] \label{cor:ALEtoE}
 Let $\Phi_i: \Omega_0 \to \Omega_i$ be a family of $C^{k}$-diffeomorphisms which converges to $\Phi:\Omega_0 \to \Omega$ in $C^k$. We define the maps
 \begin{align*}
  E_i: W^{k,p}(\Omega;\R^n) \to W^{k,p}(\Omega_i; \R^n); u \mapsto u \circ \Phi \circ \Phi_i^{-1}
 \end{align*}
 Then for any sequence $(u_i)_{i \in \N}$ with $u_i \in W^{k,p}(\Omega_i;\R^n)$ and $u \in W^{k,p}(\Omega;\R^n)$ we have that $u_i \to u$ in the ALE sense (i.e.\ $u_i \circ \Phi_i \to u\circ \Phi$ in $W^{k,p}(\Omega_0;\R^n)$) if and only if $u_i \to u$ in the sense of E-convergence.
\end{cor}

\begin{proof}
 The proof immediately follows from the fact that $ v \mapsto v\circ \Phi$ induces a linear isomorphism as well as a close look at the norms of the different functions which only differ by a factor depending on $D\Phi D\Phi_i^{-1}$ which will converge to $1$.
\end{proof}

\begin{rem}
 The converse does not hold. This is trivially clear from the fact that the maps $E_i$ in the E-convergence are rather arbitrary. But it also holds in the sense that there are meaningful situations where sequences $E$-converge using standard Sobolev-spaces, but there is no comparable ALE-result. 
 
 In particular, since $E$-convergence does not require the operators $E_i$ to be invertible, it is possible to employ it in the case of changing topology. For example take a sequence of annuli, $\Omega_i := B_1 \setminus \overline{B_{1/i}}$ and $E_i: W^{1,2}(B_1) \to W^{1,2}(\Omega_i)$ to be the restriction operators. Then what one would naturally call convergence of $u_i \in W^{1,2}(\Omega_i)$ to a limit $u\in W^{1,2}(B_1)$ is reasonably covered by the corresponding E-convergence. At the same time though, $\Omega_i$ and $B_1$ are of different topology and thus there can be no diffeomorphisms to a common reference domain.
 
 At the same time though, this strength of the E-convergence also results in a weakness. Due to the potential information loss of the operators and essentially taking the limit as a reference, it always has to be formulated with respect to a specific sequence of domains and all operators have to be constructed with that in mind. In contrast, ALE-convergence can be used in a context, where the limit-domain is not yet known or variable itself, e.g.\ when considering time-dependent domains.
\end{rem}

\subsection{Relations between the different notions and zero-extension convergence}

Since we established ALE-convergence as the weaker concept, we will now  study the connection between $E$-convergence and zero-extensions.

\begin{thm}[From zero-extension to $E$-convergence] \label{zeroToEconvergence}
  Let $(\Omega_i)_{i \in \N}$ be a sequence of bounded domains that converges to a domain $\Omega$ in the Hausdorff-sense and satisfies the uniform Lipschitz-graph condition. Let $\mathrm{Ext}: W^{k,p}(\Omega) \to W^{k,p}(\R^n)$ be an extension operator. We define
  \begin{align*}
   E_i: W^{k,p}(\Omega) \to W^{k,p}(\Omega_i); \quad u \mapsto \mathrm{Ext}(u) |_{\Omega_i}
  \end{align*}
  as the connecting system. Then a sequence with $u_i \in W^{k,p}(\Omega_i)$ converges to $u \in W^{k,p}(\Omega)$ in the sense of $E$-convergence if and only if it converges in the zero-extension sense.
\end{thm}

\begin{proof}
 Let us begin with the case $k =0$. By splitting the domain we get:
 \begin{align*}
  &\phantom{{}={}}\|u_i - E_i u\|_{L^p(\Omega_i)}^p = \int_{\Omega_i} |u_i - \mathrm{Ext}(u)|^p\\
  &= \int_{\R^n} | \tilde u_i - \tilde u|^p + \int_{\Omega_i \setminus \Omega} (|u_i - \mathrm{Ext}(u)|^p-|u_i|^p) - \int_{\Omega \setminus \Omega_i} |u|^p\\
  &=\|\tilde u_i - \tilde u\|_{L^p(\R^n)}^p +  \int_{\Omega_i \setminus \Omega} (|u_i - \mathrm{Ext}(u)|^p-|u_i|^p) - \int_{\Omega \setminus \Omega_i} |u|^p
 \end{align*}
 We thus only have to prove that these error-terms vanish in the limit. For the last term that follows directly from the fact that the convergence implies $|\Omega\setminus \Omega_i| \to 0$. For the second term we have by a standard identity
 \begin{align*}
  | |u_i - \mathrm{Ext}(u)|^p-|u_i|^p | \leq c(\mathrm{Ext}(u)) |u_i|^{p-1}
 \end{align*}
 where $c(\mathrm{Ext}(u)) \in L^p(\Omega)$. Then since $|u_i|^{p-1}$ is equiintegrable, there can be no concentrations. Thus the term vanishes as $|\Omega_i\setminus \Omega| \to 0$ as well.

For $k >0$, we simply repeat this calculation with $D^{k}u_i$ and $D^{k} \mathrm{Ext}(u)$ respectively.
\end{proof}

The converse also holds, however we have to be careful to restrict ourselves to ``reasonable'' connecting systems $E_i$, as the general definition of $E$-convergence does not require the $\Omega_i$ to be in any way related to the limit $\Omega$. We choose to enforce this connection here by requiring functions to be related on compact sets that are contained in all domains and have them avoiding concentrations, but there are other valid approaches.

\begin{thm}[From $E$-convergence to zero extension]
  Let $(\Omega_i)_{i \in \N}$ be a sequence of bounded domains that converges to a domain $\Omega$ in the Hausdorff-sense and satisfies the uniform Lipschitz-graph condition. Let $(E_i)_{i \in \N}$ be a connecting system such that for every domain $U$ which satisfies $U \subset \subset \Omega_i$ for all $i>i_0$ and for any $u \in W^{k,p}(\Omega)$, we have
  \begin{align} \label{EtoZeroRigidityCondition}
   \left\| u|_U - E_i(u)|_U \right\|_{W^{k,p}(U)} \to 0
  \end{align}
  and which is equiintegrable in the sense that for any $u \in W^{k,p}(\Omega)$ and any $\varepsilon > 0$ and $l\leq k$, there exists $\delta > 0$ such that for all $i \in \N$ and any $V\subset \Omega_i$ with $|V| < \delta$ we have $\int_V |D^l E_i(u)|^p < \varepsilon$.
  Then a sequence $(u_i)_{i \in \N}$ with $u_i \in W^{k,p}(\Omega_i)$ converges to $u \in W^{k,p}(\Omega)$ in the sense of E-convergence if and only if it converges in the zero-extension sense.
\end{thm}

\begin{proof}
 Fix $U \subset \subset \Omega$. Then by the definition of Hausdorff-convergence it also satisfies $U \subset \subset \Omega_i$ for all $i$ large enough. As in the proof of Theorem \ref{zeroToEconvergence}, it is enough to consider the case $k=0$ and then rewrite this proof with the corresponding derivatives. We now consider 
 \begin{align*}
  \|u_i - E_i u\|_{L^p(\Omega_i)}^p  =\|u_i  - E_i u\|_{L^p(U)}^p +  \int_{\Omega_i \setminus U} |u_i-E_i u|^p 
 \end{align*}
 and
 \begin{align*}
  \|\tilde u_i - \tilde u\|_{L^p(\R^n)}^p = \|u_i - u\|_{L^p(U)}^p +  \int_{\Omega \setminus U} |\tilde u_i-\tilde u|^p. 
 \end{align*}
 By \eqref{EtoZeroRigidityCondition}, the two first terms on the right hand side converge to the same limit. What is left, is to deal with the two error terms. 
 
 Assume that $u_i$ converges to $u$ in the zero-extension sense. Then the second error term vanishes trivially. For the first term, using the equiintegrability, we can pick $\Omega_i \setminus U$ small enough such that $\int_{\Omega_i \setminus U} |E_i u |^p < \varepsilon$ and the same is true for $\int_{\Omega_i \setminus U} |u_i |^p$ by the fact that $\tilde{u}_i$ converges in $L^p(\R^n)$.
 
 Conversely, if we assume that $u_i$ $E$-converges to $u$, then the first vanishes trivially. For the second we can proceed similarly, and choose $\Omega_i \setminus U$ small enough such that $\int_{\Omega_i \setminus U} |u |^p < \varepsilon$. Now assume $\delta$ and $\varepsilon$ as in the equiintegrability-condition and let $|\Omega_i \setminus U|< \delta$. If there is a (sub)sequence of $i$ such that $\int_{\Omega_i \setminus U} |u_i |^p > 2\varepsilon$, then $\int_{\Omega_i \setminus U} |u_i - E_i u |^p > \varepsilon$ by the equiintegrability, which is a contradiction to $E$-convergence.
\end{proof}

As a direct consequence, we also get the following:

\begin{cor}[Zero-extension implies ALE-convergence]
  Let $\Phi_i: \Omega_0 \to \Omega_i$ be a family of $C^{k}$-diffeomorphisms which converges to $\Phi:\Omega_0 \to \Omega$ in $C^k$. Then a sequence $(u_i)_{i \in \N}$ with $u_i \in W^{k,p}(\Omega_i)$ converges to $u \in W^{k,p}(\Omega)$ in the zero-extension sense if and only it converges in the corresponding ALE-sense.
\end{cor}

\begin{proof}
 It is possible to prove this directly. However we note that the connecting system constructed in Corollary \ref{cor:ALEtoE}, also satisfies the assumptions of the previous theorem, proving that in this case all notions coincide.
\end{proof}

\section{Conclusion and outlook}

We have introduced the concept of strong and weak zero-extension convergence for functions on changing domains and shown its usefulness and versatility. In particular we have seen that there are quite general compactness properties which are of direct importance for applications.

The intent however was to mainly give a concise, self-contained introduction. There are several adjacent topics we omitted in favor of brevity and are planning to elaborate on in future works instead. These include in particular the following:

\begin{itemize}
 \item \textbf{Time dependent domains:} For many applications, such as fluid-structure interaction, the domains depend on a time-parameter and are determined by the problem itself. As a result one needs to deal with sequences of time-dependent families of domains and their convergence. While the basic ideas are similar to those discussed here, the resulting time-dependent spaces will have a Bochner-type structure, whose topology will be of independent interest. 
 \item \textbf{The case $p=1$:} The case $p=1$ is of separate special interest. Not only does it lack reflexiveness, but it also directly connected to the space $BV$ of functions of bounded variation, which in turn is natural for zero-extension, as it can deal with the resulting jumps directly.
 \item \textbf{Boundary values for less regular domains:} So far we have only given a very brief treatment of boundary values in the case of sufficiently ``nice'', parametrized boundaries. Yet by simple completion of the smooth functions, one can define a notion of trace-space for much more irregular boundaries. The convergence of these spaces under changing domains then poses another natural avenue of studies.
\end{itemize}

\bibliographystyle{alpha} 
\bibliography{Changing_Domains}

\end{document}